\documentclass[reqno]{amsart}
\usepackage{fullpage}
\usepackage[english]{babel}
\usepackage{amsmath}
\usepackage{amssymb}
\usepackage{selinput}
\SelectInputMappings{eacute = {é}}
\usepackage{verbatim}
\usepackage{bbm}
\usepackage{enumitem}
\usepackage{latexsym}
\usepackage{color}
\usepackage{amsfonts}
\usepackage{mathrsfs}
\usepackage{esint}
\usepackage[hidelinks]{hyperref}
\usepackage{cite}
\usepackage{csquotes}

\theoremstyle{plain}
\newtheorem{prop}{Proposition}[section]
\newtheorem{thm}{Theorem}[section]
\newtheorem{lem}{Lemma}[section]
\newtheorem{rem}{Remark}[section]
\newtheorem{exa}{Example}[section]
\newtheorem{definition}{Definition}[section]

\newcommand{\R}{\mathbb{R}}
\DeclareMathOperator{\sign}{sign}   
\DeclareMathOperator{\supp}{supp}
\DeclareMathOperator{\Conv}{Conv}

\begin{document}
\allowdisplaybreaks

\title{The one-sided Lipschitz condition in the follow-the-leader approximation of scalar conservation laws}

\date{}
\author{M. Di Francesco and G. Stivaletta}
\address{Marco Di Francesco and Graziano Stivaletta - DISIM - Department of Information Engineering, Computer Science and Mathematics, University of L'Aquila, Via Vetoio 1 (Coppito) 67100 L'Aquila (AQ) - Italy}

\email{marco.difrancesco@univaq.it, graziano.stivaletta@graduate.univaq.it}

\begin{abstract}
    We consider the follow-the-leader particle approximation scheme for a $1d$ scalar conservation law with nonnegative $L^\infty_c$ initial datum and with a $C^1$ concave flux, which is known (see \cite{DiFrancescoRosini}) to provide convergence towards the entropy solution $\rho$ to the corresponding Cauchy problem. We provide two novel contributions to this theory. First, we prove that the one-sided Lipschitz condition satisfied by the approximating density $\rho^n$ proven in \cite{DiFrancescoRosini} is a \enquote{discrete version of an entropy condition}; more precisely, under fairly general assumptions on $f$ (which imply concavity of $f$) we prove that the continuum version $\left(f(\rho)/\rho\right)_x\leq 1/t$ of said condition allows to select a unique weak solution, despite $\left(f(\rho)/\rho\right)_x\leq 1/t$ is apparently weaker than the classical Oleinik-Hoff one-sided Lipschitz condition $f'(\rho)_x\leq 1/t$. Said result relies on an improved version of Hoff's uniqueness proof in \cite{Hoff}. A byproduct of it is that the entropy condition is \emph{encoded} in the particle scheme \emph{prior} to the many-particle limit, which was never proven before. Second, we prove that in case $f(\rho)=\rho(A-\rho^\gamma)$ the one-sided Lipschitz condition proven in \cite{DiFrancescoRosini} can be \emph{improved} to a discrete version of the classical (and \enquote{sharp}) Oleinik-Hoff condition. In order to make the paper self-contained with respect to \cite{DiFrancescoRosini}, we provide proofs (in some cases \enquote{alternative} ones) of all steps of the convergence of the particle scheme.
\end{abstract}

\keywords{nonlinear conservation laws, follow-the-leader approximation, one-sided Lipschitz condition, uniqueness of entropy solutions}
\subjclass[2010]{Primary: 35L65; 35A24; 35Q70. Secondary: 35F55; 35A35; 65N75; 90B20}

\maketitle

\section{Introduction}

The concept of \emph{entropy solution} for a scalar conservation law
\begin{equation}\label{eq:main_intro}
    \rho_t + f(\rho)_x = 0
\end{equation}
is a classic topic in the analysis of nonlinear PDEs, dating back to the pioneering works of Oleinik \cite{oleinik} and Kru\v{z}kov \cite{kruvzkov}. Roughly speaking, an entropy solution is a distributional solution to \eqref{eq:main_intro} in the $L^\infty_{x,t}$ space satisfying the additional distributional inequality
\begin{equation}\label{eq:kruzkov_intro}
   \eta_k(\rho)_t + q_k(\rho)_x \leq 0,
\end{equation}
where
\begin{equation}\label{eq:entropy_entropy_flux}
  \eta_k(\rho)=|\rho-k|,\qquad q_k(\rho)=(f(\rho)-f(k))\sign(\rho-k),  
\end{equation}
and $k$ is an arbitrary real number. Here $f$ is assumed to be locally Lipschitz continuous. 

It is well known that the concept of entropy solution is necessary in order to single out a unique weak solution to the Cauchy problem for \eqref{eq:main_intro} on $(x,t)\in \R \times [0,+\infty)$ with a given initial condition $\bar\rho\in L^\infty(\R)$. Such a fundamental fact was proven by Kru\v{z}kov in the multi-dimensional case $x\in \R^d$ in \cite{kruvzkov}. More precisely, \cite{kruvzkov} shows that there exists no more than one entropy solution to said Cauchy problem. For a thorough introduction to the subject, we refer to \cite{bressan_book} and \cite{dafermos_book} and the references therein.

\bigskip
Previous to Kruzkov's work, Oleinik in \cite{oleinik} provided an apparently different formulation of the entropy condition in the one-dimensional case $x\in \R$ under the further assumption that the \emph{flux} function $\rho\mapsto f(\rho)$ is $C^2$ and satisfies $f''>0$.
Such a formulation reads
\begin{equation}\label{eq:oleinik1}
   \rho_x \leq \frac{1}{(\min f'') t} \qquad \hbox{in $\mathcal{D}'(\R\times [0,+\infty))$}.
\end{equation}
A significant extension of \eqref{eq:oleinik1} to a more general case was provided by Hoff in \cite{Hoff}, in which it was proven (among other things) that the sharp version of \eqref{eq:oleinik1} in case of $f\in C^1$ and $f$ is either convex, concave, or linear, reads
\begin{equation}\label{eq:hoff1}
   f'(\rho)_x \leq \frac{1}{t} \qquad \hbox{in $\mathcal{D}'(\R\times [0,+\infty))$}.
\end{equation}
More precisely, \cite{Hoff} proves that if $f\in C^1$ then the condition \eqref{eq:hoff1} singles out a unique weak solution to the Cauchy problem for \eqref{eq:main_intro} with an $L^\infty$ initial condition \emph{if and only if} $f$ is either convex, concave, or linear.

The (distributional) one-sided Lipschitz conditions \eqref{eq:oleinik1} and \eqref{eq:hoff1} are easily interpreted as admissibility conditions for shock-wave solutions, in that they force solutions to avoid 
\enquote{non-physical} jumps. As an example, consider Burger's equation with $f(\rho)=\rho^2/2$, in which \eqref{eq:oleinik1} applies with $\min f'' = 1$. Such a condition allows for decreasing jumps, whereas increasing jumps may only occur at $t=0$ and they are smoothed to a continuous profile at positive times. In fact, the (uniform in $x$) one-sided, pointwise control of the $x$-slope in condition \eqref{eq:hoff1} forces $L^\infty$ entropy solutions to be also locally $BV$ in space. In this sense, \eqref{eq:oleinik1}-\eqref{eq:hoff1} may be interpreted as a \emph{smoothing effect}.

Condition \eqref{eq:hoff1} has, in fact, a more refined interpretation. The quantity  $1/t$ on the right-hand side of \eqref{eq:hoff1} is sharp, and achieved on \emph{rarefaction wave} profiles of the form
\[\rho(x,t)=R(\xi),\qquad \xi = \frac{x}{t},\]
with $R$ a differentiable profile. In this case, \eqref{eq:main_intro} clearly implies $f'(R(\xi)) = \xi$, and hence $f'(\rho(x,t))=x/t$, which yields the equality sign in \eqref{eq:hoff1}. Hence, condition \eqref{eq:hoff1} says that the slope of an entropy solution $\rho$ (locally) achieves its maximum on rarefaction wave profiles.

\bigskip
The resolution of the Cauchy problem for \eqref{eq:main_intro} in the entropy sense can be performed in many ways: by adding an artificial vanishing viscosity on the right-hand side of \eqref{eq:main_intro} in order to deal with parabolic equations (see e.g. \cite{oleinik} in the scalar case, \cite{kruvzkov} for general $L^\infty$ solutions in several dimensions, \cite{bardos} for initial boundary value problems and \cite{bianchinibressan} for strictly hyperbolic systems); via wave-front tracking algorithms, first introduced in \cite{dafermos}, which consist in the approximation of initial conditions with Riemann-type data and in the explicit resolution of wave interactions in the admissible sense (see e.g. \cite{bressan_book} and \cite{holdenrisebro_book} for a general overview and application to systems and multidimensional case); via nonlinear semigroup theory and Crandall-Liggett formula based on $L^1$ contraction (see e.g. \cite{crandall} in the scalar case and \cite{crandall_liggett} in a more general framework); via kinetic formulation (see e.g. \cite{perthame_tadmor} for the scalar case and \cite{lions_perthame_tadmor} for the multidimensional one); via relaxation schemes (see e.g. \cite{liu_2}, \cite{jin_xin}); via numerical algorithms (see e.g. \cite{glimm}, \cite{liu} for random choice method and \cite{godunov}, \cite{friedrichs_lax}, \cite{diperna} for finite-difference schemes), just to mention some. We refer to \cite[Section 6.9]{dafermos_book} for a detailed list of references about the above mentioned methods. 

\bigskip
In \cite{DiFrancescoRosini} a new method for the resolution of the Cauchy problem for \eqref{eq:main_intro} in the entropy sense was proposed, which works in one-space dimension, with non-negative initial data in $L^1\cap L^\infty(\R)$, under the assumption that the map
\[[0,+\infty)\ni \rho \mapsto v(\rho):=f(\rho)/\rho\]
is monotone. Said method can be seen as a \enquote{deterministic particle approximation} scheme, in that the entropy solution to the Cauchy problem is obtained as a \enquote{mean-field limit} of a system of interacting particles obeying to a system of ordinary differential equations. The latter is a \enquote{discrete} approximation of the Lagrangian evolution law $\dot{x}=v(\rho)$ encoded in the continuity equation \eqref{eq:main_intro}.

Let us sketch said approximation procedure. For simplicity, we assume here and throughout the paper that $v$ is monotone decreasing, symmetric statements hold in the increasing case, we omit the details. Given $\overline{\rho}\in L^1\cap L^\infty(\R)$, $\overline{\rho}\geq 0$, we consider a suitable \enquote{atomisation} of $\overline{\rho}$ for large $n\in \mathbb{N}$, namely a suitable set of ordered particles $\bar{x}_0,\ldots,\bar{x}_n\in \R$ such that the piecewise constant function
\[\overline{\rho}^n(x)=\sum_{k=0}^{n-1}\frac{1}{n(\bar{x}_{k+1}-\bar{x}_k)}\mathbf{1}_{[\bar{x}_k,\bar{x}_{k+1})}(x)\]
converges to $\overline{\rho}$ as $n\to +\infty$ in a sense to be specified below. We then consider the \emph{follow-the-leader} system,
\begin{equation}\label{eq:FTL_intro}
    \begin{cases}
    \displaystyle{\dot{x}_k=v\left(\frac{1}{n(x_{k+1}-x_k)}\right),\qquad k=0,\ldots,n-1},\\
    \displaystyle{\dot{x}_n= v(0)}.
    \end{cases}
\end{equation}
As $v$ is monotone decreasing, \eqref{eq:main_intro} can be seen as a first order model for \emph{traffic flow}, in which the vehicles' speed decreases with respect to the density of the vehicles $\rho$. System \eqref{eq:FTL_intro} is a discrete version of the Lagrangian law $\dot{x}=v(\rho)$ encoded in the continuity equation \eqref{eq:main_intro}, in which $\rho$ is computed as a ratio $[\hbox{mass}]/[\hbox{distance}]$, each particle is assumed to have mass $1/n$, and the distance is computed locally, at each particle $x_k$, in the \emph{positive} direction (which is consistent with the model assumption that vehicles adapt their speed according to the distance from the preceding vehicle). 

The main result in \cite{DiFrancescoRosini} states that the discrete density
\[\rho ^n(x,t)=\sum_{k=0}^{n-1}\frac{1}{n(x_{k+1}(t)-x_k(t))}\mathbf{1}_{[x_k(t),x_{k+1})(t)}(x)\]
converges strongly in $L^1_{loc}(\R\times [0,+\infty))$ towards the unique entropy solution to \eqref{eq:main_intro} with $\overline{\rho}$ as initial condition. The result is proven in two separate assumption frameworks, always assuming $\overline{\rho}\in L^1\cap L^\infty$ and $\overline{\rho}\geq 0$: 
\begin{itemize}
    \item [(i)] in case $v$ is monotone decreasing and locally Lipschitz and the initial datum $\overline{\rho}$ has bounded variation,
    \item [(ii)] in case $v\in C^1$ and strictly decreasing and the map $\rho\mapsto \rho v'(\rho)$ is non-increasing with no extra assumptions on $\overline{\rho}$.
\end{itemize}
The monotonicity of $\rho\mapsto \rho v'(\rho)$ has the following interpretation. By formally writing \eqref{eq:main_intro} as
\[\rho_t + v(\rho)\rho_x +\rho v'(\rho) \rho_x = 0\]
and recalling the concept of \emph{material derivative}
\[\frac{D \rho}{Dt} = \rho_t + v(\rho)\rho_x,\]
the continuity equation \eqref{eq:main_intro} can be formally written in the Lagrangian form
\[\frac{D \rho}{Dt} + \rho v'(\rho) \rho_x = 0,\]
so that the \enquote{Lagrangian characteristic speed} $\rho v'(\rho)$ plays the role of the first derivative of a Lagrangian flux. Hence, the monotonicity of $\rho\mapsto \rho v'(\rho)$ can be seen as a sort of \enquote{convexity of the Lagrangian flux}.

In case (i), the convergence result follows essentially by proving that the total variation of $\rho^n(\cdot,t)$ does not increase in time, similarly to what happens in other approximation procedures (cf. the wave-front tracking algorithm). In case (ii), the key estimate is a one-sided control of the difference quotient for $v(\rho^n(x_k(t),t))$, more precisely the estimate
\begin{equation}\label{eq:discrete_oleinik_intro1}
   \frac{v(\rho^n(x_{k+1}(t),t))-v(\rho^n(x_k(t),t))}{x_{k+1}(t)-x_k(t)}\leq \frac{1}{t},
\end{equation}
which yields, in particular, a uniform-in-$n$ local $BV$ estimate for $v(\rho^n(\cdot,t))$. In both (i) and (ii), the consistency of the approximation scheme is then obtained by proving that the limit of $\rho^n$ satisfies Kruzkov's entropy condition \eqref{eq:kruzkov_intro}-\eqref{eq:entropy_entropy_flux} in a distributional sense. It is worth mentioning at this stage that such a strategy does not allow to detect a \enquote{discrete analogue} of the entropy condition (which in turns happens, for example, in the wave-front tracking approximation), because the consistency with Kruzkov's condition \eqref{eq:kruzkov_intro}-\eqref{eq:entropy_entropy_flux} is obtained only for large $n$. 

\bigskip
A natural question arises at this stage: does the discrete density $\rho^n$ satisfy a discrete analogue of the entropy condition \emph{prior} to sending $n\rightarrow +\infty$? In fact, \eqref{eq:discrete_oleinik_intro1} seems to be a discrete analogue of the \emph{continuum} condition
\begin{equation}\label{eq:continuum_oleinik_intro1}
    v(\rho)_x \leq \frac{1}{t}\,,
\end{equation}
which, in turns, seems to be a good candidate to select a unique weak solution in the continuum case, because it only allows for increasing jumps ($v$ is monotone decreasing). Therefore, our previous question implies the next one: is condition  \eqref{eq:continuum_oleinik_intro1} enough to single out a unique weak solution to the Cauchy problem for  \eqref{eq:main_intro}? In case of positive answer, \eqref{eq:discrete_oleinik_intro1} would be a discrete analogue of the entropy condition in the follow-the-leader approximation scheme.

Now, recalling
\[f'(\rho)=v(\rho)+\rho v'(\rho),\]
the assumptions in (ii) (namely both $v(\rho)$ and $\rho v'(\rho)$ being monotone decreasing) imply $f$ is strictly concave. In this case, a sufficient condition to characterise entropy solutions is the one by Hoff \eqref{eq:hoff1}. Although similar to it,  \eqref{eq:discrete_oleinik_intro1} looks different from a discrete version of \eqref{eq:hoff1}. The latter should rather look like
\begin{equation}\label{eq:discrete_oleinik_intro2}
   \frac{f'(\rho^n(x_{k+1}(t),t))-f'(\rho^n(x_k(t),t))}{x_{k+1}(t)-x_k(t)}\leq \frac{1}{t}. 
\end{equation}

\bigskip
The main goal of this paper is finding an answer to the above questions and clarifying the role of discrete one-sided Lipschitz conditions in the convergence and consistency (in the entropy sense) of the follow-the-leader approximation scheme.

\bigskip
\begin{itemize}
    \item 
Our primary goal is to highlight the role of the discrete one-sided Lipschitz condition \eqref{eq:discrete_oleinik_intro1} in the convergence of the follow-the-leader scheme \eqref{eq:FTL_intro} towards entropy solutions. More precisely, we shall prove that, under a fairly general set of assumptions on $v$ and with the extra assumption that the initial condition is in $BV$, condition \eqref{eq:discrete_oleinik_intro1} allows to prove that the limit of $\rho^n$ is an entropy solution without passing through the distributional formulation in \eqref{eq:kruzkov_intro}-\eqref{eq:entropy_entropy_flux}. This fact has several positive repercussions: 
\begin{itemize}
 \item [(i)] as mentioned above, condition \eqref{eq:discrete_oleinik_intro1} somehow plays the role of a \emph{discrete version of the entropy condition} satisfied by the follow-the-leader particle approximation scheme, a novelty in the literature;
    \item [(ii)] the primary role played by the \enquote{one-sided transport} nature encoded in \eqref{eq:FTL_intro} becomes evident in the consistency of the scheme, a factor that remains somewhat hidden in the consistency proof based on Kruzkov's condition \eqref{eq:kruzkov_intro}-\eqref{eq:entropy_entropy_flux} in \cite{DiFrancescoRosini}; in this sense, it becomes clear that the \emph{entropy condition} in the scheme \eqref{eq:FTL_intro} is encoded in the choice of considering a  \enquote{forward approximation} of the Lagrangian law $\dot{x}=v(\rho)$ due to $v$ being monotone decreasing;
        \item [(iii)] the proof of the consistency of the approximation scheme gets significantly shortened and simplified.
\end{itemize}
Given that \eqref{eq:continuum_oleinik_intro1} is \emph{weaker} than \eqref{eq:hoff1}, the main technical difficulty to prove our result relies in proving 
that the continuum condition  \eqref{eq:continuum_oleinik_intro1} is enough to single out a \emph{unique} weak solution to the Cauchy problem of \eqref{eq:main_intro}. We have found one proof of that statement in the book \cite{evans_PDEs}, which works in the specific case of $f$ being uniformly convex (or concave). We shall achieve our goal by slightly improving Hoff's uniqueness proof in \cite{Hoff}.
An issue arises with respect to the regularity of the initial datum: in order to prove the above uniqueness result we shall need the initial datum to be in $BV$, which allows to achieve the initial condition in a strong $L^1$ sense, which is used in Hoff's proof. Our uniqueness result is proven in Theorem \ref{prop4}. The convergence of the scheme in this general case is resumed in Theorem \ref{thm1}.

\bigskip
\item
We then show that in the particular case 
\[v(\rho)=A-B\rho^\gamma\,,\qquad A,B,\gamma>0\,,\]
the follow-the-leader approximation \eqref{eq:FTL_intro} features an \emph{improved version of the one-sided Lipschitz estimate}, namely,  \eqref{eq:discrete_oleinik_intro2} holds in this case. Since \eqref{eq:discrete_oleinik_intro2} mimics \eqref{eq:hoff1} in the many particle limit, we need no $BV$ assumption on the initial condition to prove uniqueness of the limit, as \eqref{eq:hoff1} is already known to be equivalent to Kruzkov's condition \eqref{eq:kruzkov_intro}-\eqref{eq:entropy_entropy_flux}, and we can take advantage of Chen's and Rascle's result \cite{chen_rascle}, in which a mere continuity in the sense of measures near $t=0$ is enough to achieve uniqueness. The improved one-sided estimate is contained in Theorem \ref{prop:improved}. The convergence of the scheme in this particular case is stated in Theorem \ref{thm2}.
\end{itemize}

We emphasise that our proofs of the convergence of the particle scheme are carried out entirely without using Kruzkov's entropy condition \eqref{eq:kruzkov_intro}-\eqref{eq:entropy_entropy_flux}.

In order to make the paper self-contained, we shall provide here some of the technical results included in \cite{DiFrancescoRosini}, with the additional goal to improving their presentation. For example, we shall improve the proofs of the discrete maximum principle and of the discrete one-sided Lipschitz condition proven in \cite{DiFrancescoRosini} (more precisely, Lemmas 1 and 6 in \cite{DiFrancescoRosini}) by clarifying the assumptions on $v$ (implicitly assumed in the proof of \cite[Lemma 1]{DiFrancescoRosini}) and by using a regularised version of the positive part in \cite[Lemma 6]{DiFrancescoRosini}, which makes the proof easier to read. Moreover, our proof of the discrete maximum principle is a direct one, alternative to the \enquote{reductio ad absurdum} provided in \cite{DiFrancescoRosini}.

\section{Preliminaries and statement of the results}\label{sec:preliminaries}

\subsection{Setting of the problem}

We consider the Cauchy problem
\begin{equation}
\label{s1}
\begin{cases} \rho_t+f(\rho)_x=0, &\quad (x,t)\in \mathbb{R}\times(0,+\infty),
\\
\rho(x,0)=\overline{\rho}(x), &\quad x\in \mathbb{R},
\end{cases}
\end{equation} 
with the notation
\[f(\rho):=\rho v(\rho).\]  
Throughout the paper, we shall assume the following basic conditions on $v$ and $\overline{\rho}$:
\begin{itemize}
\item[{\bf{(I)}}] $\overline{\rho}	\in L^1(\mathbb{R})\cap L^\infty(\mathbb{R})$ is non-negative and with compact support,
\item[{\bf{(V1)}}] $\rho \mapsto v(\rho)$ is $C^1([0,\overline{R}])$ with $v$ strictly decreasing on $[0,\overline{R}]$,
\end{itemize}
with
\[\overline{R}:=||\overline{\rho}||_{L^\infty(\mathbb{R})}.\]
We use the notation \[v_{\max}:=v(0)<+\infty. \]
Moreover, by denoting
\begin{equation}\label{eq:phi_def}
    \phi(\rho):=\rho v'(\rho),
\end{equation}
we assume that
\begin{itemize}
\item[{\bf{(V2)}}] there exists a constant $K\geq 0$ such that
\begin{equation}
\label{13}
0\leq \frac{\phi(\rho)-\phi(\sigma)}{v(\rho)-v(\sigma)}\leq K
 \quad \hbox{for all $\rho,\sigma \in [0,\overline{R}]$ with $\rho\neq \sigma$}.
\end{equation}
\end{itemize}
Since $v$ is strictly decreasing, the first inequality in \eqref{13} is equivalent to requiring 
\[[0,+\infty)\ni\rho\mapsto\phi(\rho)\in \R \quad \hbox{non-increasing.}\]
We observe that, since
\[f'(\rho)=v(\rho)+\rho v'(\rho)=v(\rho)+\phi(\rho),\]
the above assumptions {\bf{(V1)}} and {\bf{(V2)}} imply that $f'$ is strictly decreasing, and hence $f$ is \emph{strictly concave}. 

\begin{exa}
A first example of function $v$ satisfying our assumptions is $v(\rho):=v_{\max}-\rho^{\gamma}$ with $\gamma>0$. We have $v'(\rho)=-\gamma\rho^{\gamma-1}\leq 0$ for all $\rho\in [0,\overline{R}]$, $\phi(\rho)=-\gamma\rho^{\gamma}$ and hence $\phi'(\rho)= -\gamma^2 \rho^{\gamma-1}\leq 0$ for all $\rho\in [0,\overline{R}]$. In this case
\[\frac{\phi(\rho)-\phi(\sigma)}{v(\rho)-v(\sigma)}=\gamma\]
and {\bf{(V2)}} is trivially satisfied.
\end{exa}

\begin{exa}
Another example is given by $v(\rho) :=v_{\max}\biggl[\log\biggl(\dfrac{1}{\beta}\biggr)\biggr]^{-1}\log\biggl(\dfrac{1}{\rho+\beta}\biggr)$ with $0<\beta<1$. We have $v'(\rho)=-v_{\max}\biggl[\log\biggl(\dfrac{1}{\beta}\biggr)\biggr]^{-1}\dfrac{1}{\rho+\beta}\leq 0$ for all $\rho\in [0,\overline{R}]$,
$\phi(\rho)=-v_{\max}\biggl[\log\biggl(\dfrac{1}{\beta}\biggr)\biggr]^{-1}\dfrac{\rho}{\rho+\beta}$ and hence $\phi'(\rho)= -v_{\max}\biggl[\log\biggl(\dfrac{1}{\beta}\biggr)\biggr]^{-1}\!\dfrac{\beta}{\rho+\beta}\leq 0$ for all $\rho\in [0,\overline{R}]$. By denoting $A=\log\biggl(\dfrac{1}{\rho+\beta}\biggr)$ and $B=\log\biggl(\dfrac{1}{\sigma+\beta}\biggr)$, we have
\[\frac{\phi(\rho)-\phi(\sigma)}{v(\rho)-v(\sigma)}=\beta\frac{e^A- e^B}{A-B},\]
which satisfies {\bf{(V2)}} with $K=\beta e^{\overline{R}}$.
\end{exa}

\begin{rem}[Sufficient condition to have {\bf{(V2)}}] In case $v\in C^2([0,\overline{R}])$, a sufficient condition for the upper bound inequality in \eqref{13} to hold is  $\phi(\rho)$ non-increasing and the existence of $\tilde{K}\geq 0$ such that
\begin{equation}
\biggl|\frac{\rho v''(\rho)}{v'(\rho)}\biggr|\leq \tilde{K} \quad \hbox{for all $\rho \in [0,\overline{R}]$}.
\end{equation}
Indeed the monotonicity of $\phi$ implies the lower bound in \eqref{13}, while, introducing the notation
\begin{equation*}
a:=v(\rho), \quad b:=v(\sigma), \quad w:=v^{-1},
\end{equation*}
we have $\phi'(\rho)=v'(\rho)+\rho v''(\rho)$, $w'(d)=\dfrac{1}{v'(w(d))}$ and hence
\begin{equation*}
\begin{split}
\frac{\phi(\rho)-\phi(\sigma)}{v(\rho)-v(\sigma)}=\frac{\phi(w(a))-\phi(w(b))}{a-b}\leq  \sup_{d\in [0,\overline{R}]}\bigl|\phi'(w(d))w'(d)\bigr|\leq & 1+\sup_{d\in [0,\overline{R}]}\biggl|\frac{w(d)v''(w(d))}{v'(w(d))} \biggr|<\infty,
\end{split}
\end{equation*}
which proves the upper bound in \eqref{13}. Moreover, both the previous two examples satisfies this sufficient condition, respectively with $\tilde{K}=|\gamma-1|$ and $\tilde{K}=\dfrac{\overline{R}}{\beta+\overline{R}}$.
\end{rem}

Let us now introduce the two concepts of entropy solution we shall deal with.

\begin{definition}[Classical entropy solution]\label{definition_classical_entropy_solution}
Let $\overline{\rho}\in L^1\cap L^\infty(\mathbb{R})$. A function $\rho\in L^\infty([0,+\infty)\,;\,L^1\cap L^\infty(\R))$ is an \emph{entropy solution} to \eqref{s1} if $\rho$ solves the equation $\rho_t+f(\rho)_x=0$ distributionally on $\R\times (0,+\infty)$, $\rho(\cdot,t)\to \overline{\rho}$ in the weak-star measure sense as $t\searrow 0$, and 
\begin{equation}
\label{def_entropy_classical}
 f'(\rho(x,t))_x \leq \frac{1}{t}\qquad \hbox{in $\mathcal{D}'(\R\times [0,+\infty))$}.
\end{equation}
\end{definition}

\begin{definition}[Extended entropy solution]\label{definition_extended_entropy_solution}
Let $\overline{\rho}\in L^1\cap L^\infty(\mathbb{R})$. A function $\rho\in L^\infty([0,+\infty)\,;\,L^1\cap L^\infty(\R))$ is an \emph{extended entropy solution} to \eqref{s1} if $\rho$ solves the equation $\rho_t+f(\rho)_x=0$ distributionally on $\R\times (0,+\infty)$, $\rho(\cdot,t)\to \overline{\rho}$ strongly in $L^1$ as $t\searrow 0$, and there exists a positive constant $C\geq 1$ such that
\begin{equation}
\label{def_entropy_extended}
 f'(\rho(x,t))_x \leq \frac{C}{t}\qquad \hbox{in $\mathcal{D}'(\R\times [0,+\infty))$}.
\end{equation}
\end{definition}

\begin{rem}
\emph{
Definition \ref{definition_classical_entropy_solution} is the one provided by Hoff in \cite{Hoff}. Note that \eqref{def_entropy_extended} is weaker than \eqref{def_entropy_classical} because of the constant $C$, which we allow to be any constant larger or equal $1$, whereas $C=1$ in Hoff's condition. However, Definition \ref{definition_extended_entropy_solution} requires a stronger continuity assumption near $t=0$ compared to Definition \ref{definition_classical_entropy_solution}.}
\end{rem}

\subsection{The follow-the-leader approximation scheme}
Let us now introduce the follow-the-leader particle approximation. For the sake of simplicity, we suppose that the initial mass is normalised, that is $\|\overline{\rho}\|_{L^{1}(\mathbb{R})}=1$, and moreover we denote with
\[[\overline{x}_{\min},\overline{x}_{\max}] :=\mathrm{Conv}( \mathrm{supp}(\overline{\rho}))\]
the convex hull of the support of $\overline{\rho}$. 

We split the interval $[\overline{x}_{\min},\overline{x}_{\max}]$ into $n$ sub-intervals having equal mass $\ell_{n}:=1/n$. So, for a fixed $n\in \mathbb{N}$ sufficiently large, we set $\overline{x}_{0}^{n}:=\overline{x}_{\min}$, $\overline{x}_{n}^{n}:=\overline{x}_{\max}$ and we define recursively
\begin{equation*}
\overline{x}_{i}^{n}:= \sup \biggl\{ x\in \mathbb{R}: \int_{\overline{x}_{i-1}^{n}}^{x}\overline{\rho}(x)dx < \ell_{n}\biggr\} \quad \text{for}\ i\in \{1, \dots, n-1\}.
\end{equation*}
From the previous definition we immediately have that $\overline{x}_{0}^{n}<\overline{x}_{1}^{n}<\dots <\overline{x}_{n}^{n}$ and
\begin{equation*}
\int_{\overline{x}_{i-1}^{n}}^{\overline{x}_{i}^{n}}\overline{\rho}(x)dx= \ell_{n} \quad \text{for}\ i\in \{1,\dots,n-1\}.
\end{equation*}
Now, we introduce the follow-the-leader system describing the evolution of the $n+1$ particles with initial positions $\overline{x}_{i}^{n}$, $i=0,\ldots,n$. Since the velocity field is non-negative and decreases with respect to the density $\rho$, the follow-the-leader scheme should consider a \emph{forward} finite-difference approximation of the density. As a consequence, with the notation
\begin{equation*}
R_{i}^{n}(t):=\frac{\ell_{n}}{x_{i+1}^{n}(t)-x_{i}^{n}(t)},\qquad t\geq 0,\quad i\in \{0,\dots,n-1\},
\end{equation*}
the ODE system we consider is
\begin{equation}
\label{9}
\begin{cases}
\dot{x}_{i}^{n}(t)=v(R_{i}^{n}(t)), & \quad \text{for} \ i\in \{0,\dots,n-1\},
\\
\dot{x}_{n}^{n}(t)=v_{\max},
\\
x_{i}^{n}(0)=\overline{x}_{i}^{n}, & \quad \text{for} \ i\in \{0,\dots,n\}.
\end{cases}
\end{equation}
A discrete maximum principle (see \cite[Lemma 1]{DiFrancescoRosini}) ensures that particles never collide and this gives the global existence of the solution for \eqref{9}. Therefore, we can construct the time-depending piecewise constant density having support in $[x_{0}(t),x_{n}(t)]$ and given by
\begin{equation}
\label{12}
\rho^{n}(x,t):=\sum_{i=0}^{n-1}R_{i}(t)\mathbf{1}_{[x_{i}(t),x_{i+1}(t))}(x)=\sum_{i=0}^{n-1}\frac{\ell_{n}}{x_{i+1}(t)-x_{i}(t)}\mathbf{1}_{[x_{i}(t),x_{i+1}(t))}(x).
\end{equation}

\subsection{Statement of the results}
We are now ready to state our results. The first one deals with the general case of $v$ satisfying {\bf{(V1)}} and {\bf{(V2)}}. First of all, we state the following uniqueness result.

\begin{thm}
\label{prop4}
Let $\overline{\rho}\in BV(\R)$ satisfy {\bf{(I)}}. Assume $v$ satisfies {\bf{(V1)}} and {\bf{(V2)}}. Then the Cauchy problem \eqref{s1} has at most one weak solution satisfying Definition \ref{definition_extended_entropy_solution}.
\end{thm}

The proof of Theorem \ref{prop4} is provided in Section \ref{sec:uniqueness}.

Next, we state the convergence result in the general case.

\begin{thm}
\label{thm1}
Assume {\bf{(V1)}}, {\bf{(V2)}} and {\bf{(I)}} hold. Assume further that $\overline{\rho}\in BV(\R)$.  Then the approximated density $\{\rho^{n}\}_{n\in\mathbb{N}}$ defined in \eqref{12} converges, up to a subsequence, almost everywhere and in $L^{1}_{loc}$ on $\mathbb{R}\times [0,+\infty)$ to the unique \emph{extended} entropy solution to the Cauchy problem \eqref{s1} in the sense of Definition \ref{definition_extended_entropy_solution}.
\end{thm}

Our next result deals with the case
\begin{equation}\label{eq:v_special}
    v(\rho):=A-\rho^\gamma,\qquad \hbox{with $A,\gamma>0$}\,.
\end{equation}

As a first important result, we state the improved version of the one-sided Lipschitz estimate.

\begin{thm}[Improved one-sided Lipschitz condition]\label{prop:improved}
Let $v(\rho)=v_{\max}-\rho^\gamma$ with $\gamma>0$. Then, for all $t\geq 0$, for all $n\in\mathbb{N}$ and for all $i\in \{0,\dots, n-1\}$, we have
\begin{equation}
\label{1}
t\frac{v(R_{i+1}(t))-v(R_i(t))}{x_{i+1}(t)-x_{i}(t)}\leq \frac{1}{\gamma+1},
\end{equation}
which also reads
\begin{equation}
\label{1bis}
t\frac{f'(R_{i+1}(t))-f'(R_i(t))}{x_{i+1}(t)-x_{i}(t)}\leq 1.
\end{equation}
\end{thm}

The proof of Theorem \ref{prop:improved} is provided in Section \ref{sec:improved}.

Finally, we state the convergence result in case \eqref{eq:v_special}.

\begin{thm}
\label{thm2}
Assume $v$ is as in \eqref{eq:v_special}, and assume that {\bf{(I)}} holds. Then the approximated density $\{\rho^{n}\}_{n\in\mathbb{N}}$ defined in \eqref{12} converges, up to a subsequence, almost everywhere and in $L^{1}_{loc}$ on $\mathbb{R}\times [0,+\infty)$ to the unique \emph{classical} entropy solution to the Cauchy problem \eqref{s1} in the sense of Definition \ref{definition_classical_entropy_solution}. 
\end{thm}

Theorems \ref{thm1} and \ref{thm2} are proven in detail in Section \ref{sec:conclusion}, based on the estimates proven in Sections \ref{sec:proof1} and \ref{sec:improved}.

\section{Proof of the uniqueness result}\label{sec:uniqueness}

In this section we prove the uniqueness of extended entropy solutions in the sense of Definition \ref{definition_extended_entropy_solution} stated in Theorem \ref{prop4}. As mentioned in the introduction, this proof follows the lines of the uniqueness proof in \cite{Hoff}. 
We first introduce some notations.

\begin{definition}
Given a function $f$ and three distinct real numbers $a,b,c$, we define the divided differences $f[a,b]$ and $f[a,b,c]$ respectively as
\begin{equation*}
f[a,b]:=\frac{f(a)-f(b)}{a-b} \quad \text{and} \quad f[a,b,c]:=\frac{f[a,b]-f[b,c]}{a-c}.
\end{equation*}
Moreover, for $f\in C^1$ and $f\in C^2$ we define respectively the divided differences 
\begin{equation*}
f[a,a]:=f'(a) \quad \text{and} \quad f[a,b,b]:=\frac{f''(d)}{2} \text{ for some } d\in \Conv(a,b).
\end{equation*}
\end{definition}

\begin{proof}[Proof of Theorem \ref{prop4}]
Let $\rho, \tilde{\rho}$ two solutions satisfying Definition \ref{definition_extended_entropy_solution}, let $0<t_1<t_2$ arbitrarily fixed and let us denote with $e:=\rho-\tilde{\rho}$. Multiplying (in a weak sense) the conservation law in \eqref{s1} by $\phi \mathbf{1}^\sigma_{[t_1,t_2]}$ where $\phi\in C^\infty(\mathbb{R}\times\mathbb{R}_+)$ is such that $\supp \phi(\cdot,t) \cap \{t_1\leq t\leq t_2\}$ is bounded, $\mathbf{1}^\sigma_{[t_1-\sigma,t_2+\sigma]}$ is a smooth approximation of $\mathbf{1}_{[t_1,t_2]}$, and by letting $\sigma\searrow 0$ (this is a standard procedure, we omit the details), we get
\begin{equation*}
\int_{\mathbb{R}}\rho(x,t)\phi(x,t)dx\biggl]_{t_1}^{t_2}=\int_{t_1}^{t_2}\int_{\mathbb{R}}\bigl(\rho(x,t)\phi_t(x,t)+f(\rho(x,t))\phi_x(x,t)\bigr)dx dt.
\end{equation*}
Writing the same identity for $\tilde{\rho}$ and subtracting term by term, we have that $e$ satisfies
\begin{equation}
\label{3}
\int_{\mathbb{R}}e(x,t)\phi(x,t)dx\biggl]_{t_1}^{t_2}=\int_{t_1}^{t_2}\int_{\mathbb{R}}e(x,t)\bigl(\phi_t(x,t)+f[\rho(x,t),\tilde{\rho}(x,t)]\phi_x(x,t)\bigr)dx dt.
\end{equation}
We now choose a suitable function $\phi$. Let us fix a function $\psi\in C_c^\infty (\mathbb{R})$ and we define
\begin{equation*}
\psi^{\pm}(x):=\frac{\psi(x)}{2}\pm\frac{1}{2}\int_{-\infty}^{x}|\psi'(s)|ds,
\end{equation*}
which satisfy $\psi= \psi^++\psi^-$,
\begin{equation}
\label{7}
(\psi^+)'=\frac{\psi'+|\psi'|}{2}\geq 0 \quad \text{and} \quad (\psi^-)'=\frac{\psi'-|\psi'|}{2}\leq 0.
\end{equation}
Moreover, we introduce two constants $\varepsilon, \delta >0$, two mollifiers $j_\varepsilon, j_\delta$ with $j_\varepsilon\in C_c^{\infty}(\mathbb{R}_+)$, $j_\delta\in C_c^{\infty}(\mathbb{R})$, and we consider the unique solutions $\phi_{\varepsilon,\delta}^+$ and $\phi_{\varepsilon,\delta}^-$ respectively of
\begin{equation}
\label{4}
\begin{cases} 
\phi_t+\bigl((j_\varepsilon * f')(\tilde{\rho})\bigr)\phi_x=0, \quad & (x,t)\in \mathbb{R}\times (t_1,t_2),
\\
\phi(x,t_2)=(j_\delta*\psi^+)(x), \quad & x\in \mathbb{R},
\end{cases}
\end{equation}
and 
\begin{equation}
\label{6}
\begin{cases} 
\phi_t+\bigl((j_\varepsilon * f')(\rho)\bigr)\phi_x=0, \quad & (x,t)\in \mathbb{R}\times (t_1,t_2),
\\
\phi(x,t_2)=(j_\delta*\psi^-)(x), \quad & x\in \mathbb{R}.
\end{cases}
\end{equation}
By the maximum principle for linear transport equations, we notice that
\begin{equation}
\label{8}
\bigl|\bigl|\phi^{\pm}_{\varepsilon,\delta}\bigr|\bigr|_{L^\infty}\leq \bigl|\bigl|j_\delta*\psi^{\pm}\bigr|\bigr|_{L^{\infty}}=\bigl|\bigl|\psi^{\pm}\bigr|\bigr|_{L^\infty}.
\end{equation}
Furthermore, the function $\phi_{\varepsilon,\delta}:=\phi^+_{\varepsilon,\delta}+\phi^-_{\varepsilon,\delta}$ is smooth and $\supp \phi_{\varepsilon,\delta}(\cdot,t)\cap \{t_1\leq t\leq t_2\}$ is bounded, hence we can take $\phi=\phi_{\varepsilon,\delta}$ in \eqref{3} and get
\begin{equation}
\label{5}
\begin{split}
&\int_{\mathbb{R}}e(x,t_2)(j_\delta*\psi)(x)dx-\int_{\mathbb{R}}e(x,t_1)\phi_{\varepsilon,\delta}(x,t_1)dx
\\
=&\int_{t_1}^{t_2}\int_{\mathbb{R}}e(x,t)\biggl(\frac{\partial \phi_{\varepsilon,\delta}}{\partial t}(x,t)+f[\rho(x,t),\tilde{\rho}(x,t)]\frac{\partial \phi_{\varepsilon,\delta}}{\partial x}(x,t)\biggr)dx dt
\\
:=&A+B,
\end{split}
\end{equation}
with 
\begin{equation*}
A:=\int_{t_1}^{t_2}\int_{\mathbb{R}}e(x,t)\Bigl(f[\rho(x,t),\tilde{\rho}(x,t)]-(j_\varepsilon * f')(\tilde{\rho})\Bigr)\frac{\partial \phi^+_{\varepsilon,\delta}}{\partial x}(x,t)dx dt,
\end{equation*}
\begin{equation*}
B:=\int_{t_1}^{t_2}\int_{\mathbb{R}}e(x,t)\Bigl(f[\rho(x,t),\tilde{\rho}(x,t)]-(j_\varepsilon * f')(\rho)\Bigr)\frac{\partial \phi^-_{\varepsilon,\delta}}{\partial x}(x,t)dx dt.
\end{equation*}
Now, we want to estimate the norm of $\dfrac{\partial \phi^\pm_{\varepsilon,\delta}}{\partial x}$. Denoting for convenience $w^{+}:=\rho, w^-:=\tilde{\rho}$, and differentiating the equations in \eqref{4} and \eqref{6} with respect to $x$, we get
\begin{equation*}
\frac{\partial}{\partial t}\frac{\partial \phi^\pm_{\varepsilon,\delta}}{\partial x}+(j_\varepsilon * f')(w^\mp)\frac{\partial}{\partial x}\frac{\partial \phi^\pm_{\varepsilon,\delta}}{\partial x}=-\biggl(j_\varepsilon * \frac{\partial f'}{\partial x}\biggr)(w^\mp)\frac{\partial \phi^\pm_{\varepsilon,\delta}}{\partial x}.
\end{equation*}
Hence, the characteristic curve $x(t)$ passing through some point $y\in \mathbb{R}$ at time $t=t_2$ is given respectively by
\begin{equation*}
\begin{cases}
\dot{x}(t)= \bigl(j_\varepsilon * f'\bigr)(w^\mp(x(t),t)),
\\
x(t_2)=y.
\end{cases}
\end{equation*}
As a consequence the functions $\dfrac{\partial \phi^\pm_{\varepsilon,\delta}}{\partial x}$, evaluated along $x=x(t)$, satisfy the equation
\begin{equation*}
\dfrac{d}{dt}\biggl[\frac{\partial \phi^\pm_{\varepsilon,\delta}}{\partial x}(x(t),t)\biggr]=-\biggl(j_\varepsilon * \frac{\partial f'}{\partial x}\biggr)(w^\mp(x(t),t))\frac{\partial \phi^\pm_{\varepsilon,\delta}}{\partial x}(x(t),t),
\end{equation*}
whose solution is given, for any $t\in (t_1,t_2)$, by
\begin{equation*}
\begin{split}
\frac{\partial \phi^\pm_{\varepsilon,\delta}}{\partial x}(x(t),t)=& \frac{\partial \phi^\pm_{\varepsilon,\delta}}{\partial x}(y,t_2)\exp\biggl[\int_{t}^{t_2}\biggl(j_\varepsilon * \frac{\partial f'}{\partial x}\biggr)(w^\mp(x(s),s)) ds\biggr]
\\
=&(j_{\delta}*\psi^\pm)'(y)\exp\biggl[\int_{t}^{t_2}\biggl(j_\varepsilon * \frac{\partial f'}{\partial x}\biggr)(w^\mp(x(s),s)) ds\biggr].
\end{split}
\end{equation*}
Since $(j_{\delta}*\psi^\pm)'=j_{\delta}*(\psi^\pm)'$ and due to \eqref{7}, we have
\begin{equation*}
\|(j_{\delta}*\psi^\pm)'\|_{L^\infty}\leq \|(\psi^\pm)'\|_{L^\infty} \quad \text{and} \quad (j_{\delta}*\psi^-)'\leq 0\leq  (j_{\delta}*\psi^+)'.
\end{equation*}
Moreover, since \eqref{def_entropy_extended} implies
\begin{equation*}
\exp\biggl[\int_{t}^{t_2}\biggl(j_\varepsilon * \frac{\partial f'}{\partial x}\biggr)(w^\mp(x(s),s)) ds\biggr]\leq \exp \biggl[C\int_{t}^{t_2}\frac{ds}{s}\biggr]=\biggl(\frac{t_2}{t}\biggr)^C,
\end{equation*}
then it finally follows
\begin{equation}
\label{10}
\left\|\frac{\partial \phi^{\pm}_{\varepsilon,\delta}}{\partial x}(x,t)\right\|_{L^\infty}\leq \|(\psi^\pm)'\|_{L^\infty}\biggl(\frac{t_2}{t}\biggr)^C.
\end{equation}
Let us now estimate $A$ (an analogous bound can be derived also for $B$). For $\eta>0$, we introduce $f_\eta(\rho):= (j_{\eta}*f)(\rho)$, where $j_\eta$ is a mollifier such that $f''_\eta\leq 0$, $f_\eta\to f,$ and $f'_\eta\to f'$ uniformly on bounded sets as $\eta\to 0$. Then it holds
\begin{equation*}
\begin{split}
A= & \int_{t_1}^{t_2}\int_{\mathbb{R}}e(x,t)\Bigl(f[\rho,\tilde{\rho}]-f_\eta [\rho,\tilde{\rho}]\Bigr)\frac{\partial \phi^+_{\varepsilon,\delta}}{\partial x}(x,t)dx dt+\int_{t_1}^{t_2}\int_{\mathbb{R}}e(x,t)\Bigl(f_\eta
[\rho,\tilde{\rho}]-f_\eta[\tilde{\rho},\tilde{\rho}]\Bigr)\frac{\partial \phi^+_{\varepsilon,\delta}}{\partial x}(x,t)dx dt
\\
&+\int_{t_1}^{t_2}\int_{\mathbb{R}}e(x,t)\Bigl(f_\eta'(\tilde{\rho})-f'(\tilde{\rho})\Bigr)\frac{\partial \phi^+_{\varepsilon,\delta}}{\partial x}(x,t)dx dt+\int_{t_1}^{t_2}\int_{\mathbb{R}}e(x,t)\Bigl(f'(\tilde{\rho})-\bigl(j_\varepsilon * f')(\tilde{\rho})\Bigr)\frac{\partial \phi^+_{\varepsilon,\delta}}{\partial x}(x,t)dx dt,
\end{split}
\end{equation*}
where the first and the third integrals go to zero as $\eta\to 0$ by the choice of $f_\eta$ whereas, since $f''_\eta\leq 0$ and $f_\eta[\rho,\tilde{\rho}]-f_\eta [\tilde{\rho},\tilde{\rho}]=f_\eta[\rho,\tilde{\rho},\tilde{\rho}](\rho-\tilde{\rho})=e f_\eta''(\zeta)/2$ for some $\zeta\in \Conv(\rho,\tilde{\rho})$, the second integral satisfies
\begin{equation*}
\int_{t_1}^{t_2}\int_{\mathbb{R}}e(x,t)\Bigl(f_\eta[\rho,\tilde{\rho}]-f_\eta [\rho,\rho]\Bigr)\frac{\partial \phi^+_{\varepsilon,\delta}}{\partial x}(x,t)dx dt=\int_{t_1}^{t_2}\int_{\mathbb{R}}e^2(x,t)\frac{f''_{\eta}(\zeta)}{2} \frac{\partial \phi^+_{\varepsilon,\delta}}{\partial x}(x,t)dx dt\leq 0.
\end{equation*}
Therefore, applying first \eqref{10} and then \eqref{8} in \eqref{5}, we have
\begin{equation*}
\begin{split}
\int_{\mathbb{R}}e(x,t_2)(j_\delta*\psi)(x)dx\leq & \int_{\mathbb{R}}e(x,t_1)\phi_{\varepsilon,\delta}^+(x,t_1)dx+\int_{\mathbb{R}}e(x,t_1)\phi_{\varepsilon,\delta}^-(x,t_1)dx
\\
&+||e||_{L^\infty}\biggl\{o_\eta(1)+\bigl|\bigl|(\psi^+)'\bigr|\bigr|_{L^\infty}\biggl(\frac{t_2}{t_1}\biggr)^C\bigl|\bigl|f'(\tilde{\rho})-(j_\varepsilon * f')(\tilde{\rho})\bigr|\bigr|_{L^1(\Omega)}\biggr\}
\\
&+||e||_{L^\infty}\biggl\{o_\eta(1)+\bigl|\bigl|(\psi^-)'\bigr|\bigr|_{L^\infty}\biggl(\frac{t_2}{t_1}\biggr)^C\bigl|\bigl|f'(\rho)-(j_\varepsilon * f')(\rho)\bigr|\bigr|_{L^1(\Omega)}\biggr\}
\\
\leq & \ ||e(\cdot,t_1)||_{L^1} \bigl(||\psi^+||_{L^\infty}+||\psi^-||_{L^\infty}\bigr)+||e||_{L^\infty} \ o_\eta(1)
\\
&+ ||e||_{L^\infty}\biggl(\frac{t_2}{t_1}\biggr)^C\biggl\{\bigl|\bigl|(\psi^+)'\bigr|\bigr|_{L^\infty} \ \bigl|\bigl|f'(\tilde{\rho})-(j_\varepsilon * f')(\tilde{\rho})\bigr|\bigr|_{L^1(\Omega)}
\\
&\hspace{2.9cm}+\bigl|\bigl|(\psi^-)'\bigr|\bigr|_{L^\infty} \ \bigl|\bigl|f'(\rho)-(j_\varepsilon * f')(\rho)\bigr|\bigr|_{L^1(\Omega)}\biggr\},
\end{split}
\end{equation*}
where $\Omega$ is a suitable bounded set depending on $j_\varepsilon$. Therefore, if we let $\eta,\varepsilon,t_1$ and $\delta$ go to zero in this order, we finally get
\begin{equation*}
\int_{\mathbb{R}}e(x,t_2)\psi(x)dx\leq 0 \quad \text{for all } \psi\in C_c^\infty(\mathbb{R}),
\end{equation*}
which implies that $\rho(\cdot,t_2)=\tilde{\rho}(\cdot,t_2)$ almost everywhere.
\end{proof}

\section{Estimates on the particles  system in the general case}\label{sec:proof1}

In this section we provide the main technical results needed to prove the convergence results in  Theorems \ref{thm1} and \ref{thm2}. These results are already proven in \cite{DiFrancescoRosini}. As mentioned in the introduction, we shall provide alternative proofs to them in order to make the paper self contained and with the goal of partly making those proofs clearer with respect to \cite{DiFrancescoRosini}. We shall assume throughout this whole section that {\bf{(V1)}}, {\bf{(V2)}}, and {\bf{(I)}} are satisfied. Further assumptions will be stated if necessary.

We start by proving the discrete maximum principle in the spirit of \cite{DiFrancescoRosini} but following an alternative, direct proof. We assume here $t\geq 0$ is arbitrary. Clearly, the proof applies only for $t$ lying in a \enquote{local existence} time interval $[0,T)$, and as a byproduct of the result one obtains global-in-time existence for \eqref{9} and that the statement below holds for all times.

\begin{prop}[Discrete Maximum Principle]\label{prop:MP}
For all $t\geq 0$ we have
\begin{equation}\label{eq:CL_maximum}
    x_{k+1}(t)-x_k(t)\geq \min_{k=0,\ldots,n-1}(\bar{x}_{k+1}-\bar{x}_k).
\end{equation}
\end{prop}

\begin{proof}
We denote
\[\Delta_{\min}:=\min_{k=0,\ldots,n-1}(\bar{x}_{k+1}-\bar{x}_k).\]

We start by estimating the distance between the two particles $x_n$ and $x_{n-1}$. For some $t\geq 0$ we integrate the equations for $x_n$ and $x_{n-1}$ on $[0,t]$ in \eqref{9} and take their difference:
\begin{align*}
     x_n(t)-x_{n-1}(t)=&\bar{x}_n-\bar{x}_{n-1} + \int_0^t\left[v_{\max}- v\left(\frac{1}{n(x_n(s)-x_{n-1}(s))}\right)\right]ds \geq  \bar{x}_n-\bar{x}_{n-1} \geq \Delta_{\min},
\end{align*}
since $v$ is monotone non-increasing on $[0,+\infty)$ by assumption {\bf{(V1)}}. 

Now, for all $k\in \{0,\ldots,n-2\}$, we prove \eqref{eq:CL_maximum} by \enquote{backward induction}. Let $k\in \{0,\ldots,n-2\}$ and let us assume 
\begin{equation}\label{eq:MP_induction}
    \min_{t\geq 0}\bigl(x_{k+2}(t)-x_{k+1}(t)\bigr)\geq \Delta_{\min},
\end{equation}
which implies
\begin{equation}\label{eq:MP_induction2}
    R_{k+1}(t)\leq \frac{1}{n \Delta_{\min}}\qquad \hbox{for all $t\geq 0$}.
\end{equation}
We set
\[\Delta(t):=x_{k+1}(t)-x_k(t),\qquad Y(t)=\frac{1}{n\Delta(t)}-\frac{1}{n\Delta_{\min}}.\]
Recall the positive part function $(z)_+=\max\{z,0\}$
and consider its regularisation
\[
\eta_\varepsilon (z)=
\begin{cases}
0 & \hbox{for $z\leq 0$},\\
\dfrac{z^2}{2\varepsilon} & \hbox{for $0\leq z\leq \varepsilon$},\\
z-\dfrac{\varepsilon}{2} & \hbox{for $z\geq \varepsilon$}.
\end{cases}
\]
We compute
\begin{align*}
    & \frac{d}{dt}\eta_{\varepsilon}(Y(t))  = \eta'_{\varepsilon}(Y(t))\dot{Y}(t)= -\eta'_{\varepsilon}(Y(t))\frac{1}{n\Delta(t)^2}\bigl(v(R_{k+1}(t))-v(R_k(t))\bigr).
\end{align*}
Now, since $\eta'_\varepsilon(Y)$ is zero on $Y\leq 0$, the right-hand side above is non-zero only if $\Delta(t)\leq \Delta_{\min}$, which is equivalent to $R_k(t)\geq \frac{1}{n\Delta_{\min}}$. Then, \eqref{eq:MP_induction2} implies
\[R_{k+1}(t)\leq  \frac{1}{n\Delta_{\min}} \leq R_k(t),\]
which due to ${\bf{(V1)}}$ implies
\[v(R_{k+1}(t))-v(R_k(t))\geq 0,\]
and therefore $\eta_\varepsilon(Y(t))$ is non-increasing in time. By letting $\varepsilon\searrow 0$, we obtain
\[(Y(t))_+\leq (Y(0))_+ =\left(\frac{1}{n\Delta(0)}- \frac{1}{n\Delta_{\min}}\right)_+ = 0\]
by the definition of $\Delta_{\min}$. Hence, $Y(t)\leq 0$ for all $t\geq 0$, which implies 
\[\Delta(t)\geq \Delta_{\min},\]
or equivalently
\[\min_{t\geq 0}\bigl(x_{k+1}(t)-x_k(t)\bigr)\geq \Delta_{\min}.\]
This concludes the proof.
\end{proof}

The result in Proposition \ref{prop:MP} guarantees the following property 
\begin{equation}\label{eq:maximum}
    R^n_i(t)\leq R,\qquad \hbox{for all $n\in \mathbb{N}$, for all $i\in\{0,\ldots,n-1\}$ and for all $t\geq 0$},
\end{equation}
where $\overline{R}$ is defined in the assumption {\bf{(I)}} as the $L^\infty$ norm of $\overline{\rho}$. Property \eqref{eq:maximum} gives the uniform estimate
\[\sup_{t\geq 0}\|\rho^n(\cdot,t)\|_{L^\infty(\R)}\leq \overline{R}.\]

Having assumed that $\bar{\rho}$ has compact support, we immediately get uniform-in-$n$ estimate of the measure of the support of $\rho^n(\cdot,t)$, more precisely
\begin{equation}\label{eq:support}
    \mathrm{supp}(\rho^n(\cdot,t)) = [\overline{x}_{\min},\overline{x}_{\max}+tv_{\max}].
\end{equation}

We now provide the key one-sided estimate \eqref{eq:discrete_oleinik_intro1} along the lines of the one proven in \cite{DiFrancescoRosini}.

\begin{prop}[One-sided Lipschitz estimate, general case]\label{prop:oleinik1}
For all $t\geq 0$, for all $n\in \mathbb{N}$ and for all $i\in\{0,\ldots,n-1\}$, we have
\begin{equation}\label{eq:oleinik2}
   t \frac{v(R_{i+1}(t))-v(R_i(t))}{x_{i+1}(t)-x_i(t)} \leq 1.
\end{equation}
\end{prop}

\begin{proof}
Let us denote, for all $i\in\{0,\ldots,n-1\}$,
\begin{equation}
\label{eq:Di}
D_i(t):= t \frac{v(R_{i+1}(t))-v(R_i(t))}{x_{i+1}(t)-x_i(t)} = t n R_i(t) \bigl(v(R_{i+1}(t))-v(R_i(t))\bigr),
\end{equation}
with the convention $R_n(t)=0$. We shall use the ODEs
\begin{align*}
    & \dot{R}_i(t) = -n R_i(t)^2 \bigl(v(R_{i+1})-v(R_i(t))\bigr),\qquad i\in\{0,\ldots,n-2\},\\
    & \dot{R}_{n-1}(t) = -n R_{n-1}(t)^2 \bigl(v_{\max}-v(R_{n-1}(t))\bigr).
\end{align*}
We compute (we omit the time dependence for simplicity)
\begin{align*}
    & \dot{D}_{n-1}= n R_{n-1}(v_{\max}-v(R_{n-1})) + t n \left[v_{\max}-v(R_{n-1}) - R_{n-1}v'(R_{n-1})\right]\dot{R}_{n-1}\\
    & \ = n R_{n-1}(v_{\max}-v(R_{n-1})) - t n^2 \left[v_{\max}-v(R_{n-1}) - R_{n-1}v'(R_{n-1})\right] R_{n-1}^2(v_{\max}-v(R_{n-1})).
\end{align*}
Due to $v'\leq 0$ in {\bf{(V1)}}, we get
\begin{align*}
    & \dot{D}_{n-1}\leq n R_{n-1}(v_{\max}-v(R_{n-1}))\left[ 1- D_{n-1}\right].
\end{align*}
Since $D_{n-1}(0)=0$, the above estimate implies $D_{n-1}(t)\leq 1$ for all $t\geq 0$, otherwise, given $t_1>0$ the first time $t$ such that $D_{n-1}(t)=1$ and $t_2>t_1$ such that $D_{n-1}(t)>1$ on $t\in (t_1,t_2)$ (recall that $D_{n-1}$ is continuous), one gets for $t\in (t_1,t_2)$
\begin{align*}
    & D_{n-1}(t)\leq D_{n-1}(t_1) + n\int_{t_1}^t  R_{n-1}(s)\bigl(v_{\max}-v(R_{n-1}(s))\bigr)\bigl[ 1- D_{n-1}(s)\bigr] ds \leq D_{n-1}(t_1) = 1,
\end{align*}
i. e. a contradiction. Therefore, \eqref{eq:oleinik2} is proven for $i=n-1$. 

The case $i\in \{0,\ldots,n-2\}$ is proven inductively. Assume $D_{k+1}(t)\leq 1$ for all $t\geq 0$. We need to prove $D_k(t)\leq 1$ for all $t\geq 0$. We compute
\begin{align}
    & \dot{D}_k = n R_k (v(R_{k+1})-v(R_k))\left[1-D_k\right] -n R_k R_{k+1} v'(R_{k+1})D_{k+1} + n R_k^2 v'(R_k) D_k\nonumber\\
    & \ \leq  n R_k (v(R_{k+1})-v(R_k))\left[1-D_k\right] - n R_k R_{k+1}  v'(R_{k+1})+ n R_k^2 v'(R_k) D_k,\label{eq:oleinik_estimate1}
\end{align}
where we have used $D_{k+1}\leq 1$. Now, for small $\delta>0$, consider a smooth approximation $\R\ni\sigma\mapsto \eta_\delta(\sigma)$ of the positive part function $\R\ni\sigma\mapsto (\sigma)_+ = \max\{0,\sigma\}$ such that $\eta_\delta(\sigma)\to (\sigma)_+$ uniformly on $\sigma\in \R$, $\eta_\delta(\sigma)=\eta'_\delta(\sigma)=0$ for all $\sigma\leq 0$, $\eta'_\delta(\sigma)\in (0,1]$ for all $\sigma>0$, and $\sigma \eta'_\delta(\sigma)\to (\sigma)_+$ uniformly on $\sigma\in \R$. We compute, recalling \eqref{eq:phi_def} and using \eqref{eq:oleinik_estimate1}, 
\begin{align*}
     \frac{d}{dt}\eta_\delta(D_k(t)) =\eta'_\delta(D_k) \dot{D}_k \ \leq & n\eta'_\delta(D_k) R_k (v(R_{k+1})-v(R_{k}))\left[1-D_k\right]
     \\
     &- n\eta'_\delta(D_k)  R_k \phi(R_{k+1}) + n\eta'_\delta(D_k)  R_k\phi(R_k) D_k.
\end{align*}
Now, since three terms in the above right-hand side are non-zero only if $D_k\geq 0$, which is equivalent to $v(R_{k+1})\geq v(R_k)$, condition ${\bf{(V2)}}$ implies
\[\phi(R_{k+1})\geq \phi(R_k),\]
and hence
\begin{align*}
    &  \frac{d}{dt}\eta_\delta(D_k(t)) \leq n \eta'_\delta(D_k) R_k \left[ v(R_{k+1})-v(R_k)-\phi(R_{k})\right] \left[1-D_k\right].
\end{align*}
Since $[v(R_{k+1})-v(R_k)-\phi(R_{k})]\geq 0$ and due to $0\leq \eta'_\delta\leq 1$, we get
\begin{align*}
    & \frac{d}{dt}\eta_\delta(D_k(t)) \leq n R_k \left[v(R_{k+1})-v(R_k)-\phi(R_{k})\right] \left[1-\eta'_\delta(D_k) D_k\right].
\end{align*}
We claim that the above inequality implies that $(D_k(t))_+\leq 1$ for all $t\geq 0$. Suppose by contradiction that $t_1>0$ is the first time $t$ such that $(D_k(t_1))_+=1$ and $t_2>t_1$ is such that $(D_k(t))_+>1$ on $t\in (t_1,t_2)$ (recall that $D_{k}$ is continuous). Then, one gets for $t\in (t_1,t_2)$ 
\begin{align*}
    & \eta_\delta(D_k(t)) \leq \eta_\delta(D_k(t_1)) +n\int_{t_1}^{t} R_k(s) \left[v(R_{k+1}(s))-v(R_k(s))-\phi(R_{k}(s))\right]\left[1-\eta'_\delta(D_k(s))D_k(s)\right] ds\,
\end{align*}
and, by letting $\delta\searrow 0$, we obtain
\begin{align*}
     (D_k(t))_+ &\leq (D_k(t_1))_+ + n\int_{t_1}^t R_k(s) \left[v(R_{k+1}(s))-v(R_k(s))-\phi(R_{k}(s))\right]\left[1-(D_k(s))_+\right]ds\, 
    \\
    &\leq (D_k(t_1))_+=1,
\end{align*}
which is a contradiction. Hence, $D_k(t)\leq (D_k(t))_+\leq 1$ for all times, which concludes the proof.
\end{proof}

\section{Improved one-sided Lipschitz estimate}\label{sec:improved}

We now provide the proof of Theorem \ref{prop:improved}, which contains an improved version of the discrete one-sided Lipschitz condition in the special case \eqref{eq:v_special}.

\begin{proof}[Proof of Theorem \ref{prop:improved}]
With the same notation introduced in \eqref{eq:Di}, in the present case we have
\begin{align*}
 D_i(t)&= t n R_i(t) \left[R_{i}(t)^\gamma-R_{i+1}(t)^\gamma\right], &\quad i\in\{0,\ldots,n-1\*,
\\
     \dot{R}_i(t) &= -n R_i(t)^2 \left[R_{i}(t)^\gamma-R_{i+1}(t)^\gamma\right] \ \text{ and } \ t\dot{R}_{i}(t) = -D_{i}(t) R_{i}(t), &\quad i\in\{0,\ldots,n-2\},
    \\
     \dot{R}_{n-1}(t) &= -n R_{n-1}(t)^{\gamma+2}.
\end{align*}
Therefore, a simple computation implies
\begin{align*}
    & \dot{D}_{n-1}= n R_{n-1}^{\gamma+1} + (\gamma+1)t n R_{n-1}^{\gamma} \dot{R}_{n-1}= n R_{n-1}^{\gamma+1}-(\gamma+1)t n^2 R_{n-1}^{2(\gamma+1)}= n R_{n-1}^{\gamma+1}[1- (\gamma+1)D_{n-1}].
\end{align*}
Since $D_{n-1}(0)=0$, the above ODE implies $D_{n-1}(t)\leq \frac{1}{\gamma+1}$ for all $t\geq 0$, otherwise, given $t_1>0$ the first time $t$ such that $D_{n-1}(t)=\frac{1}{\gamma+1}$ and $t_2>t_1$ such that $D_{n-1}(t)>\frac{1}{\gamma+1}$ on $t\in (t_1,t_2)$ (recall that $D_{n-1}$ is continuous), one gets for $t\in (t_1,t_2)$
\begin{align*}
    & D_{n-1}(t)=D_{n-1}(t_1) + n \int_{t_1}^t  R_{n-1}(s)^{\gamma+1}[1- (\gamma+1)D_{n-1}(s)] ds \leq D_{n-1}(t_1) = \frac{1}{\gamma+1},
\end{align*}
i. e. a contradiction. Therefore, \eqref{1} is proven for $i=n-1$.

The case $i\in \{0,\ldots,n-2\}$ is proven inductively. Assume $D_{k+1}(t)\leq \frac{1}{\gamma+1}$ for all $t\geq 0$. We need to prove $D_k(t)\leq \frac{1}{\gamma+1}$ for all $t\geq 0$. We compute
\begin{align}
    & \dot{D}_k = n R_k (R_{k}^\gamma-R_{k+1}^\gamma)\left[1-D_k\right]  + \gamma n R_k R_{k+1}^\gamma  D_{k+1}-\gamma n R_k^{\gamma+1} D_{k}\nonumber\\
    & \ \leq   n R_k (R_{k}^\gamma-R_{k+1}^\gamma)\left[1-D_k\right] + \frac{\gamma}{\gamma+1} n R_k R_{k+1}^\gamma  -\gamma n R_k^{\gamma+1} D_{k},\label{eq:oleinik_estimateimpr}
\end{align}
where we have used $D_{k+1}\leq \frac{1}{\gamma+1}$. Now, for small $\delta>0$, consider a smooth approximation $\R\ni\sigma\mapsto \eta_\delta(\sigma)$ of the positive part function $\R\ni\sigma\mapsto (\sigma)_+ = \max\{0,\sigma\}$ such that $\eta_\delta(\sigma)\to (\sigma)_+$ uniformly on $\sigma\in \R$, $\eta_\delta(\sigma)=\eta'_\delta(\sigma)=0$ for all $\sigma\leq 0$, $\eta'_\delta(\sigma)\in (0,1]$ for all $\sigma>0$, and $\sigma \eta'_\delta(\sigma)\to (\sigma)_+$ uniformly on $\sigma\in \R$. We compute, using \eqref{eq:oleinik_estimateimpr}, 
\begin{align*}
     \frac{d}{dt}\eta_\delta(D_k(t)) =&\eta'_\delta(D_k) \dot{D}_k \leq  n\eta'_\delta(D_k) R_k (R_{k}^\gamma-R_{k+1}^\gamma)\left[1-D_k\right]
   +\frac{\gamma}{\gamma + 1} n\eta'_\delta(D_k)  R_k R_{k+1}^\gamma - \gamma n\eta'_\delta(D_k)  R_k^{\gamma+1} D_k.
\end{align*}
Now, we add and subtract in the above right-hand side the term
\begin{equation*}
\gamma n \eta'_\delta(D_k) R_k R_{k+1}^{\gamma} D_k
\end{equation*}
and we notice that the three terms in the above right-hand side are non-zero only if $D_k\geq 0$, which is equivalent to $R_{k}\geq R_{k+1}$. Therefore, we have
\begin{align*}
     \frac{d}{dt}\eta_\delta(D_k(t)) \leq & n\eta'_\delta(D_k) R_k (R_{k}^\gamma-R_{k+1}^\gamma)\left[1-D_k\right] +\frac{\gamma}{\gamma + 1} n\eta'_\delta(D_k)  R_k R_{k+1}^\gamma \left[1-(\gamma+1)D_k\right]
    \\
    &- \gamma n \eta'_\delta(D_k)  R_k(R_k^{\gamma} -R_{k+1}^{\gamma})D_k
    \\
    = & n \eta'_\delta(D_k) R_k \Bigl(R_{k}^\gamma-\frac{1}{\gamma+1}R_{k+1}^\gamma\Bigr)\left[1-(\gamma+1)D_k\right].
\end{align*}
Since $R_{k}^\gamma\geq R_{k+1}^\gamma\geq \frac{1}{\gamma+1}R_{k+1}^\gamma$ and $0\leq \eta'_\delta\leq 1$, we get
\begin{align*}
    & \frac{d}{dt}\eta_\delta(D_k(t)) \leq n R_k \Bigl(R_{k}^\gamma-\frac{1}{\gamma+1}R_{k+1}^\gamma\Bigr)\left[1-(\gamma+1)\eta'_\delta(D_k)D_k\right].
\end{align*}
We claim that the above inequality implies that $(D_k(t))_+\leq \frac{1}{\gamma+1}$ for all $t\geq 0$. Suppose by contradiction that $t_1>0$ is the first time $t$ such that $(D_k(t_1))_+=\frac{1}{\gamma+1}$ and $t_2>t_1$ is such that $(D_k(t))_+>\frac{1}{\gamma+1}$ on $t\in (t_1,t_2)$ (recall that $D_{k}$ is continuous). Then, one gets for $t\in (t_1,t_2)$ 
\begin{align*}
    & \eta_\delta(D_k(t)) \leq \eta_\delta(D_k(t_1)) +n\int_{t_1}^{t} R_k(s) \Bigl(R_{k}(s)^\gamma-\frac{1}{\gamma+1}R_{k+1}(s)^\gamma\Bigr)\left[1-(\gamma+1)\eta'_\delta(D_k(s))D_k(s)\right] ds\,
\end{align*}
and, by letting $\delta\searrow 0$, we obtain
\begin{align*}
     (D_k(t))_+ &\leq (D_k(t_1))_+ + n\int_{t_1}^t R_k(s) \Bigl(R_{k}(s)^\gamma-\frac{1}{\gamma+1}R_{k+1}(s)^\gamma\Bigr)\left[1-(\gamma+1)(D_k(s))_+\right]ds\, 
    \\
    &\leq (D_k(t_1))_+=\frac{1}{\gamma+1},
\end{align*}
which is a contradiction. Hence, $D_k(t)\leq (D_k(t))_+\leq \frac{1}{\gamma+1}$ for all times, which completes the proof of \eqref{1}. 

Finally, the equivalence between \eqref{1} and \eqref{1bis} is a consequence of the identity $f'(\rho)=v(\rho)+\rho v'(\rho)=v_{\max}-(\gamma+1)\rho^\gamma$, where the last equality holds by the choice of $v$.
\end{proof}

\section{Conclusion of the convergence proofs}\label{sec:conclusion}

In this section we conclude the proofs of Theorems \ref{thm1} and \ref{thm2}.
In order to make the paper self-contained, we also provide here the proof of the $1$-Wasserstein equi-continuity in time of $\rho^n$ contained in \cite[Proposition 3.4]{DiFrancescoRosini}. 

\begin{prop}\label{prop:wasserstein}
There exists a constant $C\geq 0$ independent of $n$ and of $t$ such that, for all $s,t\geq 0$,
\begin{equation}
    \label{eq:uniform_in_time}
    W_1(\rho^n(\cdot,t),\rho^n(\cdot,s))\leq C|t-s|.
\end{equation}
\end{prop}

\begin{proof}
It is well known (see e.g. \cite{villani_optimal}) that 
\[ W_1(\rho^n(\cdot,t),\rho^n(\cdot,s)) = \|X^n(\cdot,t)-X^n(\cdot,s)\|_{L^1(\R)},\]
where $X^n:[0,1]\times [0,+\infty)\to \R$ is the unique measurable function such that $X^n(\cdot,t):[0,1]\to \R$ is the inverse of $F^n(\cdot,t)$ restricted to $[x_0(t),x_n(t)]$. A simple computation shows 
\[X^n(z,t)=\sum_{i=0}^{n-1}\left[x_i(t)+(z-i/n)R_i(t)^{-1}\right]\mathbf{1}_{[i/n,(i+1)/n)}(z) + x_n(t)\mathbf{1}_{\{1\}}(z).\]
Hence,  
\begin{align*}
    & W_1(\rho^n(\cdot,t),\rho^n(\cdot,s)) \leq  \frac{1}{n}\sum_{i=0}^{n-1}|x_i(t)-x_i(s)| +\sum_{i=0}^{n-1}\left|R_i(s)^{-1}-R_i(t)^{-1}\right|\int_{i/n}^{(i+1)/n}(z-i/n) dz\\
    & \ \leq \frac{1}{n}\sum_{i=0}^{n-1}\int_s^t |v(R_i(\tau))|d\tau +\frac{1}{2 n^2}\sum_{i=0}^{n-1} \int_s^t \left|\frac{d}{d\tau}R_i(t)^{-1}\right|d\tau\\
    & \ \leq \max\{v_{\max}, |v(\bar{R})|\}|t-s| +\frac{1}{2 n}  \sum_{i=0}^{n-1}\int_s^t\left|v(R_i(\tau)-v(R_{i+1}(\tau))\right|d\tau \leq 2\max\{v_{\max}, |v(\overline{R})|\} |t-s|,
\end{align*}
which completes the proof.
\end{proof}

\begin{prop}[Strong compactness of $\rho^n$]\label{prop:compactness}
The sequence $\rho^n$ has a subsequence that converges almost everywhere on $\R\times (0,+\infty)$ and in $L^1_{loc}(\R\times (0,+\infty))$.
\end{prop}

\begin{proof}
The estimate \eqref{eq:oleinik2} implies that $v(\rho^n(\cdot,t))$ satisfies the one sided estimate
\[\bigl(v(\rho^n(x_{i+1}(t),t)-v(\rho^n(x_i(t),t))\bigr)_+\leq \frac{1}{\delta}(x_{i+1}(t)-x_i(t))\qquad \hbox{on $(x,t)\in \R\times [\delta,+\infty)$}\]
for all $\delta>0$. Moreover, \eqref{eq:maximum} implies that $v(\rho^n)$ is uniformly bounded. Hence, $v(\rho^n)$ has a uniformly (in $n$) bounded total variation on compact subsets of $\R\times (0,+\infty)$, which implies that $v(\rho^n)$ is  strongly compact in $L^1([-M,M]\times [\delta,T])$ for all $M\geq 0 $, $\delta>0$, and $T>\delta$. This follows as a consequence of \eqref{eq:uniform_in_time} and an Aubin-Lions type lemma contained in \cite{rossi_savare}, see \cite[Appendix A]{DFFR}. Since $v$ is strictly monotone, then it is invertible and $\rho^n$ is strongly compact on the same set. Choosing $M=n$, $\delta=1/n$, $T=n$, we can apply a diagonal procedure and obtain a subsequence of $\rho^n$ with the desired properties.  
\end{proof}

We observe that up to this point we never required $\overline{\rho}$ to be in $BV$. We now need this condition in order to obtain uniform $L^1$ continuity in time near $t=0$, a property that allows to prove the uniqueness of extended entropy solutions.

First of all, we recall that the follow-the-leader scheme \eqref{9} preserves the initial upper bound for the total variation.

\begin{lem}\label{lem:Bv}
Assume further that  $\overline{\rho}\in BV(\R)$. Then, 
\begin{equation}
    \label{eq:TV}
    \mathrm{TV}[\rho^n(\cdot,t)]\leq \mathrm{TV}[\overline{\rho}^n]\leq \mathrm{TV}[\overline{\rho}],\qquad \hbox{for all $t\geq 0$}.
\end{equation}
\end{lem}

\begin{proof}
Let $\eta_\delta:\R\to [0,+\infty)$ be a smooth approximation of the absolute value function, such that $\eta_\delta(\rho)\to |\rho|$ uniformly on $\R$, $\eta_\delta(\rho)$ is even, $\eta'_\delta(\rho)\in [0,1]$ for $\rho\geq 0$ and $\eta'_\delta(\rho)=1$ for $\rho\geq \delta$, $\eta''_\delta(\rho)\geq 0$ for $\rho\in \R$ and $\eta''_\delta(\rho) = 0$ for all $\rho \not\in [-\delta,\delta]$.
We define
\[\mathrm{TV}_\delta[\rho^n(t)]:=R_0(t)+R_{n-1}(t) + \sum_{k=0}^{n-2}\eta_\delta(R_{k+1}(t)-R_k(t)).\]
We compute
\begin{align}
    & \frac{d}{dt} \mathrm{TV}_\delta[\rho^n(t)] = \dot{R}_0(t) + \dot{R}_{n-1}(t) + \sum_{k=0}^{n-2}\eta'_\delta(R_{k+1}(t)-R_k(t))(\dot{R}_{k+1}(t)-\dot{R}_k(t))\nonumber\\
    & \ = \dot{R}_0(t)\bigl(1- \eta'_\delta(R_1(t)-R_0(t))\bigr)+ \dot{R}_{n-1}(t)\bigl(1+ \eta'_\delta(R_{n-1}(t)-R_{n-2}(t))\bigr)\nonumber\\
    & \ + \sum_{k=1}^{n-2}\dot{R}_k(t)\bigl(\eta'_\delta(R_k(t)-R_{k-1}(t))-\eta'_\delta(R_{k+1}(t)-R_k(t))\bigr).\label{eq:TVestimate}
\end{align}
We observe
\[\dot{R}_{n-1}(t)=-n R_{n-1}^2(t) \bigl(v_{\max}-v(R_{n-2}(t))\bigr)\leq 0,\]
which implies 
\[\dot{R}_{n-1}(t)\bigl(1+ \eta'_\delta(R_{n-1}(t)-R_{n-2}(t))\bigr)\leq 0.\]
Moreover, we have
\[\dot{R}_0(t)=-nR_0^2(t)\bigl(v(R_1(t))-v(R_0(t))\bigr).\]
Hence, since $v$ is strictly decreasing, the above term is non-positive for $R_1(t)\leq R_0(t)$. On the other hand, if $R_1(t)-R_0(t)\geq \delta$ then $\eta'_\delta(R_1(t)-R_0(t))=1$. Therefore, by considering only the range $R_1(t)-R_0(t)\in [0,\delta)$, the continuity of $v$ implies
\[\dot{R}_0(t)\bigl(1- \eta'_\delta(R_1(t)-R_0(t))\bigr)\leq n\,\, o_\delta(1),\]
where we have also used that $R_0(t)$ is uniformly bounded in time due to the discrete maximum principle.
We now consider the generic term
\begin{align*}
    &\dot{R}_k(t)\bigl(\eta'_\delta(R_k(t)-R_{k-1}(t))-\eta'_\delta(R_{k+1}(t)-R_k(t))\bigr)\\
    & \ = -n R_k(t)^2\bigl(v(R_{k+1}(t))-v(R_k(t))\bigr)\bigl(\eta'_\delta(R_k(t)-R_{k-1}(t))-\eta'_\delta(R_{k+1}(t)-R_k(t))\bigr)
\end{align*}
for all $k\in\{1,\ldots,n-2\}$. In case $R_{k+1}(t)-R_k(t)\leq -\delta$, the above term is $\leq 0$ since $\eta'_\delta(R_{k+1}(t)-R_k(t))=-1$ in that case and $v(R_{k+1}(t))\geq v(R_k(t))$ because $v$ is decreasing. If $R_{k+1}(t)-R_k(t)\geq \delta$, then $\eta'_\delta(R_{k+1}(t)-R_k(t))=1$ and $v(R_{k+1}(t))-v(R_k(t))\leq 0$, which implies the whole term is non-positive. In case $R_{k+1}(t)-R_k(y)\in(-\delta,\delta)$, then the continuity of $v$ once again implies the whole term is controlled by $n\,\, o_\delta(1)$ due to the continuity of $v$. Therefore, we integrate \eqref{eq:TVestimate} and let $\delta\searrow 0$ to obtain 
\[\mathrm{TV}[\rho^n(\cdot,t)] \leq \mathrm{TV}[\overline{\rho}^n].\]
Finally, since $\overline{\rho}\in BV$, then
\[\mathrm{TV}[\overline{\rho}^n] \leq \mathrm{TV}[\overline{\rho}],\]
because every $BV$ function has a right-continuous almost everywhere representation, and by the application of the mean value formula.
\end{proof}

\begin{prop}
\label{prop:nearzero}
Assume $\overline{\rho}$ satisfies ${\mathbf{(I)}}$ and $\overline{\rho}\in BV(\R)$.
There exists a constant $C\geq 0$ independent of $n$ and such that
\[\|\rho^n(\cdot,t)-\overline{\rho}^n\|_{L^1(\R)}\leq Ct^{1/2}.\]
\end{prop}

\begin{proof}
For a generic function $g\in BV(\R)$, we have the following inequality
\begin{equation}
    \label{eq:GN}
    \|g\|_{L^1(\R)}\leq C \mathrm{TV}[g]^{1/2} \|G\|_{L^1(\R)}^{1/2},\qquad \hbox{with}\,\,\, G(x):=\int_{-\infty}^x g(y,t) dy.
\end{equation}
Inequality \eqref{eq:GN} is a consequence of a special case of Gagliardo-Nirenberg inequality applied to a function $g\in W^{1,1}(\R)$ and then using the approximation of $BV$ functions, see \cite[Theorem 3.9, Proposition 3.7]{AFP}.
The assertion follows by applying \eqref{eq:GN} to $g=\rho^n(t)-\overline{\rho}^n$, observing that, by denoting \[G(x):=\int_{-\infty}^x \bigl(\rho^n(y)-\overline{\rho}^n(y)\bigr)dy,\]
we have
\[\|G\|_{L^1(\R)} = W_1(\rho^n(t),\overline{\rho}^n),\]
and then recalling \eqref{eq:uniform_in_time} and \eqref{eq:TV}.
\end{proof}

\bigskip
We now turn our attention to the up-to-subsequence limit $\rho$ (almost everywhere and locally in $L^1(\R\times (0,+\infty))$) obtained by strong compactness in Proposition \ref{prop:compactness}.

As a first property, $\rho$ is a weak solution to \eqref{eq:main_intro} on $\R\times (0,+\infty)$.

\begin{prop}\label{prop:weak}
For all $\varphi\in C^1_c(\R\times (0,+\infty))$, we have
\begin{equation}
    \label{eq:weak_formulation}
    \int_0^{+\infty}\int_\R \left[\rho(x,t) \varphi_t(x,t) + f(\rho(x,t))\varphi_x(x,t)\right] dx dt = 0. 
\end{equation}
\end{prop}

\begin{proof}
The strong converge of $\rho^n$ to $\rho$ almost everywhere and locally in $L^1$ easily implies
\[ \int_0^{+\infty}\int_\R \left[\rho^n(x,t) \varphi_t(x,t) + f(\rho^n(x,t))\varphi_x(x,t)\right] dx dt \to \int_0^{+\infty}\int_\R \left[\rho(x,t) \varphi_t(x,t) + f(\rho(x,t))\varphi_x(x,t)\right] dx dt\]
as $n\to +\infty$. Hence, we only need to prove that the left-hand side above converges to zero as $n\to +\infty$. A direct computation shows
\begin{align*}
    &  \int_0^{+\infty}\int_\R \left[\rho^n(x,t) \varphi_t(x,t) + f(\rho^n(x,t))\varphi_x(x,t)\right] dx dt \\
    & \ = \sum_{i=0}^{n-1}\int_0^{+\infty}R_i(t)\biggl[\int_{x_i(t)}^{x_{i+1}(t)}\varphi_t(x,t) dx +  v(R_i(t))\left[\varphi(x_{i+1}(t),t)-\varphi(x_i(t),t)\right]\biggr]dt\\
    & \ = \sum_{i=0}^{n-1}\int_0^{+\infty}R_i(t)\bigl[-\dot{x}_{i+1}(t)\varphi(x_{i+1}(t),t)+\dot{x}_i(t)\varphi(x_i(t),t) + v(R_i(t))\left[\varphi(x_{i+1}(t),t)-\varphi(x_i(t),t)\right]\bigr]dt\\
    & \ \ +\sum_{i=0}^{n-1}\int_0^{+\infty}n R_i(t)^2 (\dot{x}_{i+1}(t)-\dot{x}_i(t))\int_{x_i(t)}^{x_{i+1}(t)} \varphi(x,t) dx dt\\
    & \ = \sum_{i=0}^{n-1}\int_0^{+\infty}R_i(t)\left[v(R_{i}(t))-v(R_{i+1}(t))\right]\biggl[\varphi(x_{i+1}(t),t) -\fint_{x_i(t)}^{x_{i+1}(t)} \varphi(x,t) dx \biggr]dt.
\end{align*}
Since $\varphi\in C^1_c$, we obtain
\begin{align*}
    & \left|\varphi(x_{i+1}(t),t) -\fint_{x_i(t)}^{x_{i+1}(t)} \varphi(x,t) dx \right|\leq \fint_{x_i(t)}^{x_{i+1}(t)}\left|\varphi(x_{i+1}(t),t)- \varphi(x,t)\right| dx\\
    & \ \leq C (x_{i+1}(t)-x_i(t)),
\end{align*}
for some positive constant $C$ only depending on $\varphi$. Therefore, for some finite $T\geq 0$ such that $\mathrm{supp}(\varphi)\subset \R\times [0,T]$, we get
\begin{align*}
    & \left|\int_0^{+\infty}\int_\R \left[\rho^n(x,t) \varphi_t(x,t) + f(\rho^n(x,t))\varphi_x(x,t)\right] dx dt\right|\\
    & \ \leq \frac{C T}{n} \sup_{t\in [\delta,T]}\sum_{i=0}^{n-1}\Bigl|v|_{\mathrm{supp}(\varphi(\cdot,t))}(R_{i+1}(t))-v|_{\mathrm{supp}(\varphi(\cdot,t))}(R_i(t))\Bigr| = \frac{C T}{n} \sup_{t\in [\delta,T]}\mathrm{TV}\bigl(v|_{\mathrm{supp}(\varphi(\cdot,t))}(\rho^n(\cdot,t))\bigr).
\end{align*}
Since the total variation of $v(\rho^n(\cdot,t))$ restricted to the compact set $\bigcup_{t\in [\delta,T]} \mathrm{supp}(\varphi(\cdot,t))$ is uniformly bounded with respect to $n$ and $t\in [\delta,T]$, the last term above tends to zero as $n\to +\infty$, and the assertion is proven.
\end{proof}

A key property satisfied by the limit $\rho$ is given by the following proposition. Note that the assumption of $\overline{\rho}\in BV(\R)$ is not required here.
\begin{prop}
\label{prop3}
Assume that {\bf{(V1)}}, {\bf{(V2)}} and {\bf{(I)}} hold. Then the limit $\rho$ satisfies
\begin{equation*}
v(\rho(x+z,t))-v(\rho(x,t))\leq \frac{z}{t} \quad \text{for all } x\in \mathbb{R}, z>0 \text{ and } t>0.
\end{equation*}
In case $v$ is as in \eqref{eq:v_special}, then the limit $\rho$ satisfies
\begin{equation*}
f'(\rho(x+z,t))-f'(\rho(x,t))\leq \frac{z}{t} \quad \text{for all } x\in \mathbb{R}, z>0 \text{ and } t>0.
\end{equation*}
\end{prop}
\begin{proof}
Recall the one-sided estimate \eqref{eq:oleinik2}
\begin{equation}
\label{16}
t \frac{v(\rho^n(x_{i+1}(t),t))-v(\rho^n(x_{i}(t),t))}{x_{i+1}(t)-x_i(t)}\leq 1\quad \text{for} \ i\in \{0,\dots,n-1\}, \ t> 0.
\end{equation}
Let us now denote with $z:=x_{i+1}(t)-x_i(t)$ and $x:=x_i(t)$. Multiplying \eqref{16} by any test function $\varphi\in C_c^\infty(\mathbb{R}\times \mathbb{R}_+)$ with $\varphi\geq 0$ and integrating on $\mathbb{R}\times \mathbb{R}_+$, we get
\begin{equation*}
\int_{\mathbb{R}_+}\int_{\mathbb{R}} t\frac{v(\rho^n(x+z,t))-v(\rho^n(x,t))}{z}\varphi(x,t) dx dt\leq \int_{\mathbb{R}_+}\int_{\mathbb{R}} \varphi (x,t) dx dt,
\end{equation*}
that is
\begin{equation*}
\int_{\mathbb{R}_+}\int_{\mathbb{R}}t v(\rho^n(x,t))\frac{\varphi(x,t)-\varphi(x-z,t)}{z} dx dt\geq -\int_{\mathbb{R}_+}\int_{\mathbb{R}} \varphi (x,t) dx dt.
\end{equation*}
Letting first $z\to 0$ and then $n\to +\infty$, by the strong convergence of $\rho^n$ to $\rho$ we get
\begin{equation*}
\int_{\mathbb{R}_+}\int_{\mathbb{R}}v(\rho(x,t))\varphi_x(x,t) dx dt\geq - \int_{\mathbb{R}_+}\int_{\mathbb{R}}\frac{1}{t} \varphi (x,t) dx dt
\end{equation*}
and finally, applying Lemma \ref{prop2}, we get the conclusion. The statement for $v$ in \eqref{eq:v_special} follows similarly by substituting $v$ by $f'$.
\end{proof}

Now we show that, in the general case for $v$, the limit $\rho$ satisfies the entropy condition in Definition \ref{definition_extended_entropy_solution}.
\begin{prop}
\label{Prop1}
Assume that {\bf{(V1)}}, {\bf{(V2)}}, and {\bf{(I)}} hold. Then there exists a positive constant $C\geq 1$, depending only on $v$ and $\overline{\rho}$, such that the limit $\rho$ satisfies \eqref{def_entropy_extended}.
\end{prop}
\begin{proof}
We first notice that, by Lemma \ref{prop2}, proving \eqref{def_entropy_extended}  is equivalent to show that
\begin{equation*}
f'(\rho(x+z,t))-f'(\rho(x,t))\leq \frac{Cz}{t} \quad \text{for all } x\in \mathbb{R}, z>0 \text{ and } t>0.
\end{equation*}
Let us fix $x\in\mathbb{R}, z>0$ and $t>0$ and let us suppose that $\rho(x+z,t)\neq \rho(x,t)$ (otherwise the result is trivially true). Then it holds
\begin{equation*}
f'(\rho(x+z,t))-f'(\rho(x,t))=\frac{f'(\rho(x+z,t))-f'(\rho(x,t))}{v(\rho(x+z,t))-v(\rho(x,t))}\bigl(v(\rho(x+z,t))-v(\rho(x,t))\bigr),
\end{equation*}
where we already know, by Proposition \ref{prop3}, that
\begin{equation*}
v(\rho(x+z,t))-v(\rho(x,t))\leq \frac{z}{t}.
\end{equation*}
Now, assumption {\bf{(V2)}} implies
\begin{equation*}
\frac{f'(\rho(x+z,t))-f'(\rho(x,t))}{v(\rho(x+z,t))-v(\rho(x,t))} = 1 + \frac{\phi(\rho(x+z,t))-\phi(\rho(x,t))}{v(\rho(x+z,t))-v(\rho(x,t))}\leq 1 + K.
\end{equation*}
Moreover, due to {\bf{(V2)}} the above ratio is always non-negative. Therefore, 
\begin{align*}
    & f'(\rho(x+z,t))-f'(\rho(x,t)) \leq (1+K)\frac{z}{t},
\end{align*}
which completes the proof.
\end{proof}

We now prove that the limit $\rho$ has the required $L^1$ continuity near $t=0$ needed to prove uniqueness in the case of extended entropy solutions.

\begin{lem}\label{lem:nearzero}
Under the assumptions {\bf{(I)}} and $\overline{\rho}\in BV$, the limit $\rho$ satisfies
\[\lim_{t\searrow 0} \|\rho(\cdot,t)-\overline{\rho}\|_{L^1(\R)} = 0.\]
\end{lem}

\begin{proof}
By the second inequality in \eqref{eq:TV} we have that, up to a subsequence, $\overline{\rho}^n$ converges almost everywhere to $\overline{\rho}$. By Fatou's lemma then we have
\[ \|\rho(\cdot,t)-\overline{\rho}\|_{L^1(\R)} \leq \liminf_{n\to +\infty} \|\rho^n(\cdot,t)-\overline{\rho}^n\|_{L^1(\R)},\]
and finally Proposition \ref{prop:nearzero} implies the assertion.
\end{proof}

We now complete the proofs of the two convergence Theorems \ref{thm1} and \ref{thm2}.

\begin{proof}[Proof of Theorem \ref{thm1}]
Proposition \ref{prop:compactness} gives, up to a subsequence, strong convergence of $\{\rho^n\}_{n\in\mathbb{N}}$ almost everywhere on $\R\times (0,+\infty)$ and locally in $L^1$. Lemma \ref{lem:nearzero} gives the $L^1$ continuity of $\rho$ near zero. Proposition \ref{prop:weak} proves $\rho$ satisfies the conservation law in the sense of distributions. Moreover, Proposition \ref{Prop1} proves that $\rho$ satisfies the one-sided Lipschitz condition \eqref{def_entropy_extended}. Hence, $\rho$ is an extended entropy solution in the sense of Definition \ref{definition_extended_entropy_solution}. Finally, Proposition \ref{prop4} provides uniqueness.
\end{proof}

\begin{proof}[Proof of Theorem \ref{thm2}]
Proposition \ref{prop:compactness} gives, up to a subsequence, strong convergence of $\{\rho^n\}_{n\in\mathbb{N}}$ almost everywhere on $\R\times (0,+\infty)$ and locally in $L^1$. Proposition \ref{prop:weak} proves $\rho$ satisfies the conservation law in the sense of distributions. The $1$-Wasserstein continuity proven in Proposition \ref{prop:wasserstein} shows $\rho$ is continuous at $t=0$  in the $1$-Wasserstein metric. Moreover, the second statement in Proposition \ref{prop3} proves that $\rho$ satisfies the one-sided Lipschitz condition \eqref{def_entropy_classical}. Hence, $\rho$ is a classical entropy solution in the sense of Definition \ref{definition_classical_entropy_solution}. Uniqueness follows from the uniqueness result in \cite{chen_rascle} and to the fact that \eqref{def_entropy_classical} is equivalent to Kruzkov's condition \eqref{eq:kruzkov_intro}-\eqref{eq:entropy_entropy_flux}.
\end{proof}

\appendix
\section{A technical lemma}

\begin{lem}
\label{prop2}
Let $\rho\in L^\infty([0,+\infty)\,;\,L^1\cap L^\infty(\R))$, $g\in C([0,\overline{R}])$ and $C\geq 0$. Then the following three properties are equivalent:
\begin{align}\label{20}
&{\bf{A)}}\quad  g(\rho(x,t))_x\leq \frac{C}{t} \quad \hbox{in  $\mathcal{D}'(\R\times[0,+\infty))$.}
\\&\label{21}
{\bf{B)}}\quad \int_0^{+\infty}\int_{\mathbb{R}} t\,g(\rho(x,t))\varphi_x(x,t) dx dt\geq - C \int_{0}^{+\infty}\int_{\mathbb{R}} \varphi (x,t) dx dt \quad \text{for all  $\varphi\in C_c^\infty(\mathbb{R}\times \mathbb{R}_+)$ with $\varphi\geq 0$.}
\\&\label{22}
{\bf{C)}}\quad g(\rho(x+z,t))-g(\rho(x,t))\leq \frac{Cz}{t} \quad \text{for all } x\in \mathbb{R}, z>0 \text{ and } t>0.
\end{align}
\end{lem}
\begin{proof}
The equivalence between ${\bf{A)}}$ and ${\bf{B)}}$ is trivial. Let us now fix $a, z\in\mathbb{R}$ with $z>0$ and we choose $\varphi(x,t)=\phi_\varepsilon(x)\psi(t)$ where, for $\varepsilon\in (0,z)$,  $\phi_\varepsilon(x)$ is a non-negative $C^{\infty}$ function such that
\begin{equation*}
\phi_\varepsilon(x):=\begin{cases} 1 \quad &\text{for } x\in [a+\varepsilon,a+z-\varepsilon],
\\
0<\phi_\varepsilon(x)<1 \quad &\text{for } x\in (a,a+\varepsilon)\cup (a+z-\varepsilon,a+z),
\\
0 \quad &\text{otherwise,}
\end{cases}
\end{equation*} 
and $\partial_x \phi_\varepsilon(x)\to \delta_{a}(x)-\delta_{a+z}(x)$ in $\mathcal{D}'(\mathbb{R})$ as $\varepsilon\to 0$, whereas $\psi\in C_c^\infty(\mathbb{R}_+)$ with $\psi\geq 0$. With this choice of $\varphi$ and letting $\varepsilon\to 0$, then from \eqref{21} we get
\begin{equation*}
\int_0^{+\infty} t \Bigl[g(\rho(a+z,t))-g(\rho(a,t))\Bigr]\psi(t) dt \leq C z\int_0^{+\infty}\psi (t) dt,
\end{equation*}
that is
\begin{equation*}
\int_{0}^{+\infty}  \Bigl[t\left(g(\rho(a+z,t))-g\bigl(\rho(a,t)\bigr)\right)-Cz\Bigr]\psi(t) dt\leq 0.
\end{equation*} 
Finally, since $\psi\in C_c^{\infty}(\mathbb{R}_+)$ with $\psi \geq 0$ is arbitrary, then \eqref{22} holds and ${\bf{B)}}\Longrightarrow {\bf{C)}}$.
To obtain the implication ${\bf{C)}}\Longrightarrow {\bf{B)}}$ we multiply \eqref{22} by $t/z$ times a non-negative test function $\varphi\in C^\infty_c(\mathbb{R}\times [0,+\infty))$ and integrate on $\R\times [0,+\infty)$. A change of variable implies
\[\int_0^{+\infty} \int_\R  t g(\rho(x,t)) \frac{\varphi(x-z,t)-\varphi(x,t)}{z} dx dt \leq C\int_0^{+\infty} \int_\R\varphi(x,t) dx dt
\]
and ${\bf{B)}}$ follows by letting $z\searrow 0$.
\end{proof}

\section*{Acknowledgments}
We thank M. D. Rosini for his useful suggestions on the writing of this manuscript. Part of this work was carried out during the visit of MDF to King Abdullah
University of Science and Technology (KAUST) in Thuwal, Saudi Arabia. MDF is deeply grateful
for the warm hospitality by people at KAUST, for the excellent scientific environment, and for
the support in the development of this work. 
MDF acknowledges support from the \enquote{InterMaths} project of the DISIM Department at the University of L'Aquila.

{\small\bibliography{references}
\bibliographystyle{abbrv}}
\end{document}